\documentclass[12pt]{amsart}

\usepackage{amsmath}
\usepackage{amsfonts}
\usepackage{amssymb}
\usepackage{latexsym}
\usepackage{graphicx}
\usepackage{enumerate}
\usepackage{multirow}

\usepackage{float}
\usepackage{graphicx}
\usepackage{caption}
\usepackage{subcaption}

\newtheorem{theorem}{Theorem}
\newtheorem{proposition}[theorem]{Proposition}
\newtheorem{lemma}[theorem]{Lemma}
\newtheorem{cor}[theorem]{Corollary}
\newtheorem{remark}{Remark}
\newtheorem{definition}{Definition}
\newtheorem{assumption}{Assumption}

\usepackage{mathtools}

\makeatletter
\DeclareRobustCommand\widecheck[1]{{\mathpalette\@widecheck{#1}}}
\def\@widecheck#1#2{
    \setbox\z@\hbox{\m@th$#1#2$}%
    \setbox\tw@\hbox{\m@th$#1%
       \widehat{%
          \vrule\@width\z@\@height\ht\z@
          \vrule\@height\z@\@width\wd\z@}$}%
    \dp\tw@-\ht\z@
    \@tempdima\ht\z@ \advance\@tempdima2\ht\tw@ \divide\@tempdima\thr@@
    \setbox\tw@\hbox{%
       \raise\@tempdima\hbox{\scalebox{1}[-1]{\lower\@tempdima\box
\tw@}}}%
    {\ooalign{\box\tw@ \cr \box\z@}}}
\makeatother

\newcommand{\ep}{\varepsilon}
\newcommand{\om}{\omega}

\DeclareMathOperator{\sech}{sech}

\DeclareMathOperator{\sinc}{sinc}

\newcommand{\ds}{\displaystyle}

\newcommand{\be}{\begin{equation}}
\newcommand{\ee}{\end{equation}}
\newcommand{\bes}{\begin{equation*}}
\newcommand{\ees}{\end{equation*}}
\newcommand{\mand}{\quad \text{and}\quad}

\newcommand{\R}{{\bf{R}}}

\newcommand{\Z}{{\bf{Z}}}

\newcommand{\al}{\alpha}

\newcommand{\F}{{\mathfrak{F}}}

\renewcommand{\L}{{\mathcal{L}}}
\newcommand{\A}{{\mathcal{A}}}
\newcommand{\B}{{\mathcal{B}}}

\renewcommand{\P}{{\mathcal{P}}}
\newcommand{\Q}{{\mathcal{Q}}}
\renewcommand{\O}{{\mathcal{O}}}

\newcommand{\N}{{\mathcal{N}}}

\newcommand{\ulambda}{\underline{\lambda}}

\newcommand{\M}{{\mathcal{M}}}

\renewcommand{\tilde}{\widetilde}
\renewcommand{\hat}{\widehat}

\title{Coherent structures in long range FPUT lattices, Part I: Solitary Waves}
\author{Udoh Akpan}\address{Mathematics Department, West Texas A\&M, uakpan@wtamu.edu}
\author{J. Douglas Wright}\address{Department of Mathematics, Drexel University, jdw66@drexel.edu}

\begin{document}
\maketitle

\section{Preamble}
We put $W^{s,p}:=W^{s,p}(\R)$, the usual Sobolev spaces as defined in (say) \cite{E}. 
We also set $H^s:=W^{s,2}$, $L^2:=H^0$, $L^\infty:=W^{0,\infty}$ and $E^s := H^s \cap\{\text{even functions}\}$. We write ``$f(x)$ is $C^r$ on $I$'' if $f^{(r)}(x)$ is uniformly continuous on the set $I \subset \R$. We write ``$f(x)$ is $C^{r,q}$ on $I$'' if $f^{(r)}(x)$ is uniformly $q$-H\"older 
continuous on the set $I \subset \R$.
We use the following conventions for the Fourier transform and its inverse: $\ds\F[f](k) := \hat{f}(k) = {1 \over 2\pi} \int_\R e^{-ikx} f(x) dx$ and $\ds \F^{-1}[\hat{f}](x) := \int_\R e^{ikx} \hat{f}(k) dk$. 

\section{Introduction}
 We consider a general version of the Fermi-Pasta-Ulam-Tsingou (FPUT) lattice with long range interactions.  Specifically:
 \begin{equation}\label{motion}
    \ddot{u}_{j} =  \sum_{m\ge1} \Phi_{m}'(u_{j+m} - u_{j}) - \Phi_{m}'(u_{j} - u_{j-m}).
\end{equation}
Here $j \in \Z$, $t \in \R$, $u_j(t) \in \R$ and $\Phi_m : \R \to \R$. The usual interpretation is that these are the equations of motion for a mechanical system consisting of infinitely many particles, indexed by $j$, arranged on a line and interacting pairwise. The position of the $j^{th}$ particle at time $t$ is $u_j(t)$ and
 $\Phi_m(r)$ gives the 
potential energy associated with the interaction between a particle and the particles $m$ spots down the line, in either direction. In this setting $\eqref{motion}$ is simply Newton's law. There are multiple other interpretations, ranging from power transmission and molecular dynamics to automotive traffic, but here we stick with the usual mechanical point of view.

Our goal in this paper, and its sequel,  is to understand the existence of traveling wave solutions to systems of the form \eqref{motion}.
There are many situations where such solutions of \eqref{motion} have been found:
\begin{itemize}
\item When $\Phi_m (r) =0$ for $m > 1$, \eqref{motion} corresponds to a classical FPUT problem
and there is a rather vast literature on traveling wave solutions: see \cite{T,FW,FP1,V,VT,SK1,SK2,SWsol}.
\item The ``next nearest neighbor'' (NNN) problem, where $\Phi_m(r) =0$ for $m > 2$, has been studied in \cite{FW,VZ}.
\item The case where 
$\Phi_m (r)=0$ for $m \ge N$, for some $N>1$, (that is to say interactions which have long, but finite, range) was studied in \cite{HML}\footnote{\cite{HML}, by Michael Herrmann and Alice Mikikits-Leitner, served as a major inspiration for this work.  Michael Herrmann, a phenomenal mathematician and even more phenomenal colleague, passed away in Summer 2024. This article is dedicated to his memory.}. We will call this situation the ``finite range'' (FR)
problem.
\item Putting $\Phi_m(r) = 1/r^{a}$ corresponds to an analysis of generalized Calogero-Moser lattices. Solitary waves have recently been constructed for $a\in(4/3,3)$ in \cite{IP1, IP2}.
\end{itemize}
Our work in this paper extends the results of \cite{VZ,HML} to interactions with infinite range, {\it i.e.}  to situations where $\Phi_m \ne 0$ for infinitely many $m$. In particular we develop necessary conditions on the potentials $\Phi_m$ which lead to the existence of solitary wave solutions. A notable application will be to generalized Calogero-Moser lattices with $a >3$.
In a subsequent paper we will construct periodic and generalized traveling wave solutions.

The lattice \eqref{motion} is in equilibrium when $u_j(t) = r_* j$, where $r_*$ is a constant\footnote{This constant $r_*$ can be any real number (or at least any real number for which $r_* m$ is in the domain of $\Phi_m'$).  
Many works, though hardly all, make us of the fact that a simple coordinate change and modification of the potentials can put $r_*=0$. We don't do that and instead take $r_*$ as  being some specified, fixed value throughout.}.
We make the following assumption on the regularity of the potentials $\Phi_m$ so that we can expand them in a series about this steady state.
\begin{assumption} \label{phass}
There exists $\delta_*>0$ so that, for all $m \ge 1$, $\Phi_m(r_*m + \eta)$ is $C^{3,1}$
on  $|\eta| \le m \delta_*$.
\end{assumption}
Consequently we have
\be\label{expando}
\Phi'_m(r_*m + \eta) =\varsigma_m+ \al_m \eta +  \beta_m \eta^2 + \Psi'_m(\eta)
\ee
with $\Psi'_m(\eta) = \O(\eta^3)$. Precisely, $\Psi'_m(0) = 0$ and there are constants $\gamma_m \ge 0$ so that
\be\label{Psi est}
|\eta| \le m \delta_* \implies |\Psi'_m(\eta)| \le  \gamma_m |\eta|^3 \mand |\Psi''_m(\eta)| \le 3\gamma_m |\eta|^2. 
\ee
Of course $\varsigma_m:=\Phi'(r_*m)$, 
$\al_m := \Phi''(r_* m)$ and $\beta_m := \Phi'''(r_*m)/2$.   Likewise
$
\gamma_m :=  {\textrm{Lip }} \Phi'''_m/6,
$
where by ${\textrm{Lip }} \Phi'''_m$ we mean the Lipschitz constant of $\Phi'''_m(r_* m + \eta)$ on $|\eta|\le m \delta_*$.
The constants $\varsigma_m$ will play almost no role.

Many aspects of our main results are best phrased in terms of the 
 the dispersion relation, which we compute presently.

\subsection{The dispersion relation, phase speed and classification thereof}
Linearizing \eqref{motion} about the equilibrium $r_* j$ gives
\be \label{linear}
\ddot{u}_j = \sum_{m\ge1}  \al_m(u_{j+m} -2 u_j + u_{j-m}).
\ee
Plugging the plane wave {\it Ansatz} $u_j(t) = e^{i (k j - \omega t)}$ into \eqref{linear} and some routine algebra/trigonometry results in
the dispersion relation
\be\label{dr}
\om^2=\theta(k):=
\sum_{m\ge1} 4 \al_m \sin^2(mk/2).
\ee
We also define
\be\label{lambda}
\lambda(k) := {\theta(k)\over k^2}=\sum_{m\ge1} \al_m m^2 \sinc^2(mk/2),
\ee
along with
\be
c_0^2 := {\lambda(0)} = \sum_{m \ge 1} \al_m m^2.
\ee
$\sqrt{\lambda(k)}$ is the phase speed and  $c_0$ is variously referred to as the ``speed of sound'' or the ``critical speed.''

We make a few observations about $\theta(k)$: (a) it is $2\pi$-periodic in the wavenumber $k$, (b) it is even in $k$, and (c) $\theta(0) = 0$.
Indeed, given that we have not specified the coefficients $\al_m$ at this stage, we see something rather surprising: $\theta(k)$ can be {\it any} function meeting (a), (b) and (c). We codify this:
\begin{lemma}\label{wild}
Let $\theta(k)$ be a piecewise continuously differentiable, even, $2\pi$-periodic function which vanishes at $k=0$.  Then there are coefficients $\al_m$ so that the linear dispersion relation for \eqref{motion} is given by $\om^2 = \theta(k)$.
\end{lemma}
\begin{proof} The hypotheses on $\theta(k)$ imply that it is equal to its Fourier cosine series. That is to say $\theta(k) = \sum_{m =0}^\infty b_m \cos(mk)$ for some coefficients $b_m$. The condition that $\theta(0) =0$ tells us that $b_0 = -\sum_{m=1}^\infty b_m$ which in turn tells us that $\theta(k) = \sum_{m =1}^\infty b_m(-1+\cos(mk))$. Letting $\al_m = -b_m/2$ along with the half angle formula finishes the proof.
\end{proof}
%

We will not make direct use of Lemma \ref{wild} in what follows, but that there is so much flexibility in $\theta(k)$ means 
that it is often easier to work 
in terms of it and $\lambda(k)$ as opposed to conditions on the coefficients $\al_m$. 
We now distinguish two special types of $\lambda(k)$:
\begin{definition} \phantom{duh} \begin{itemize}
\item
We say $\lambda(k)$ is {\bf Type I} if the following hold.
\begin{enumerate}[(i)]
\item $\lambda(k)$ is bounded below.
\item $\lambda''(0) < 0$.
\item There exists $\mu_* > 0$, $k_* > 0$ and $\sigma \in (0,2]$ so that $|k| \le k_*$ implies
$$
\lambda(k) - \lambda(0) \le - \mu_* k^2 \mand
|\lambda(k) - \lambda(0) - {1 \over 2} \lambda''(0) k^2| \le \mu_* |k|^{2+\sigma}.
$$
\item $\sup_{|k| \ge k_*} \lambda(k) < \lambda(0)$.
\end{enumerate}
\item
We say $\lambda(k)$ is  {\bf Type II}  if the following hold.
\begin{enumerate}
\item $\lambda(k)$ is bounded below.
\item $\lambda''(0) > 0$.
\item $\lambda(k)$ is analytic for $|\Im(k)| < \rho$ where $\rho>0$.
\item There exist $0 \le c_-<c_0$ so that the equation 
$
\lambda(k) = c^2
$
has a unique positive solution $k=k_c$ when $c \in (c_{-},c_0)$
and $\lim_{c\to c_0^-}\lambda'(k_c) \ne 0$.
\end{enumerate}
\end{itemize}
\end{definition}
In Figure \ref{lambdafig} we sketch examples of $\lambda(k)$ for both types.  See the caption there for more details.  
We are not saying that all lattices fall into one of these two types. Indeed Lemma~\ref{wild} precludes such an easy 
categorization.  But a routine classical FPUT lattice is of Type I, as are the FR lattices studied in \cite{HML}.  And  the NNN lattices studied in \cite{VZ} are of Type II.  More on this below.

\begin{figure}
	\centering
	\begin{subfigure}{.45\textwidth}
		\includegraphics[width=\textwidth]{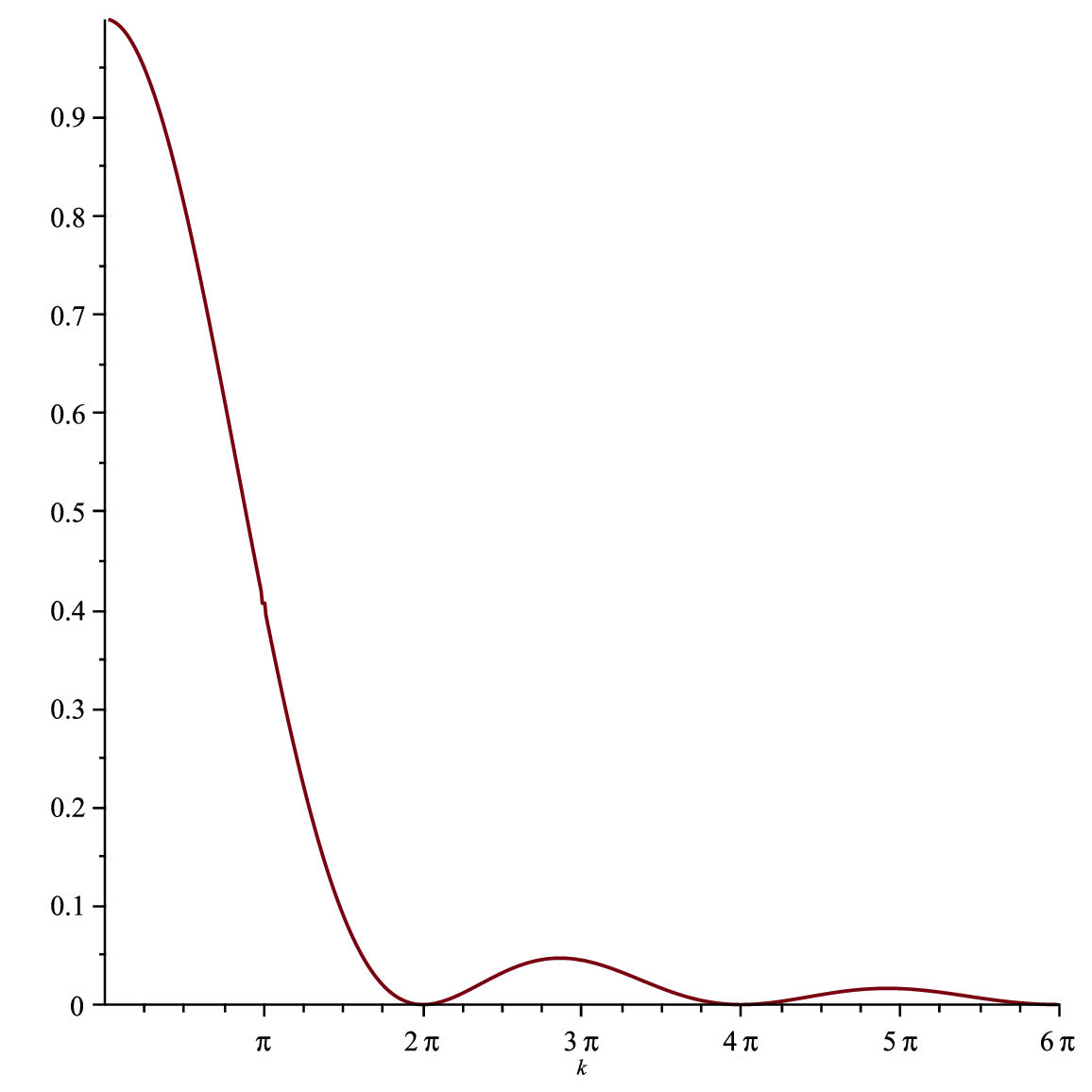}
		\caption{$\lambda(k)$ is Type I.}
	\end{subfigure}
	\begin{subfigure}{.45\textwidth}
		\includegraphics[width=\textwidth]{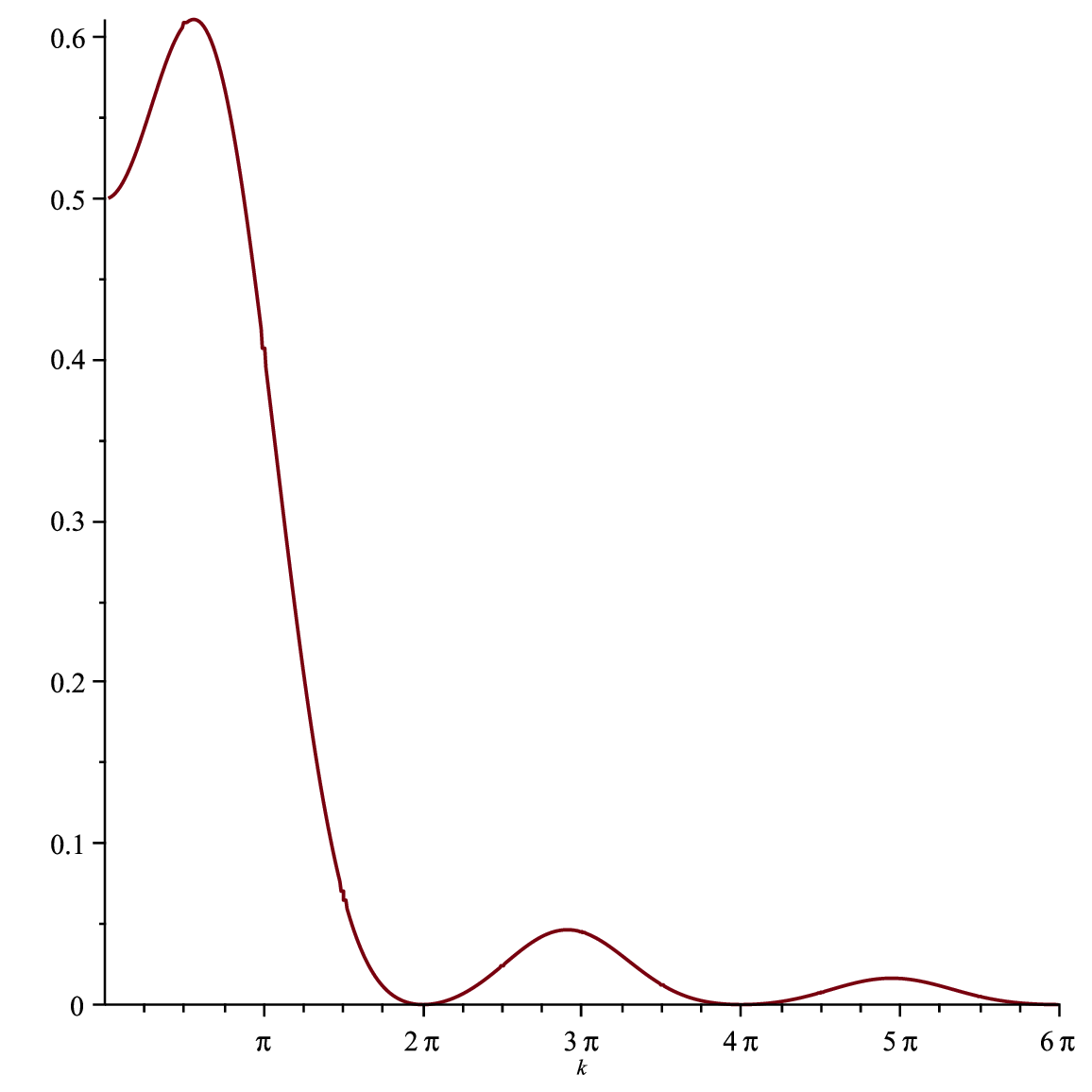}
		\caption{$\lambda(k)$ is Type II.}
	\end{subfigure}
	\caption{$\lambda(k)$ vs $k$.}\label{lambdafig}
\end{figure}

Roughly speaking, the main result of this article is that if $\lambda(k)$ is 
Type I, then \eqref{motion} possesses supersonic (that is $c>c_0$) solitary waves (this is Theorem \ref{type I sw} below).
If $\lambda(k)$ is Type II then \eqref{motion} possesses subsonic (that is $c<c_0$) spatially periodic traveling waves and nanoptera solutions; this will be the focus of
a future paper.  

The conditions in the Type I and Type II definitions could be interpreted as assumptions on $\al_m$;
we also need assumptions on $\beta_m$ and $\gamma_m$.  These are
\begin{assumption} \label{bass}
The sums 
$\ds
\sum_{m \ge 1}|\beta_m| m^{5}$ and $\ds
\sum_{m \ge 1} |\gamma_m| m^{4}$ converge.  Additionally $\ds \sum_{m \ge 1} \beta_m m^3$ is non-zero.
\end{assumption}

\subsection{The traveling wave equation and long wave limit}
We make the traveling wave {\it Ansatz}:
\be\label{TWA}
u(j,t) = r_*j+\ep U_\ep (x) \mand x:=\ep (j-c_\ep t).
\ee
Following \cite{HML}, we have incorporated the ``long wave scaling'' here, made manifest by the presence of the parameter $0< \ep \ll1$. Ultimately
we will use $\ep$ as a bifurcation parameter.  
Substitution of \eqref{TWA}  into \eqref{motion} and a bit of effort demonstrate that 
 $
W_\ep(x):=U_\ep'(x)$ and the wave-speed $c_\ep$ solve the ``nonlinear eigenvalue'' equation
\be\label{TWE1}
\ep^2 c_\ep^2 W_\ep= \sum_{m\ge1} m \A_{\ep m} [\Phi_m'(r_*m+ m \ep^2 \A_{\ep m} W_\ep)-\varsigma_m].
\ee
The maps $\A_{\ep m}$  (as in \cite{HML}) are instances of the averaging operators
$$
\A_h F(x):={1 \over h} \int_{x - h/2}^{x+h/2} F(y) dy.
$$
The derivation of \eqref{TWE1} is carried out in the proof of Lemma 2 in \cite{HML} for the finite range problem and it carries over to this case with essentially no extra work; we omit the details.
Additionally at this stage we assume that $c_\ep$ has the form
\be\label{WSA}
c^2_\ep = c_0^2 - {1 \over 2} \lambda''(0) \ep^2.
\ee
The reason for the unusual prefactor on the $\ep^2$ term will become clear below; see Remark \ref{speed scale}.

Using the expansion \eqref{expando} in \eqref{TWE1} converts it to
\be\label{TWE2}
\B_\ep W_\ep = \Q_\ep(W_\ep,W_\ep) + \ep^2 \P_\ep(W_\ep)
\ee
where
\bes\begin{split}
\B_\ep&:=\ep^{-2}\left(c_0^2-{1 \over 2} \lambda''(0)\ep^2 - \Lambda_\ep\right), \\
\Lambda_\ep &:={\sum_{m\ge1}  \al_m m^2 \A^2_{\ep m} 
},\\
\Q_\ep(V,W)&:=\sum_{m\ge1} \beta_m m^3 \A_{\ep m} [( \A_{\ep m} V)(\A_{\ep m} W)],\\
\P_\ep(W)&:={1 \over \ep^{6}} \sum_{m\ge1} m \A_{\ep m} [\Psi_m'(m \ep^2 \A_{\ep m} W)].
\end{split}\ees
Note that all of the mappings defined above respect parity in that they map even functions to even functions and so on.

In \cite{HML} it is shown that taking the Fourier transform of $\A_h F(x)$ gives
$$
\hat{\A_h F}(k) =\sinc(h k/2) \hat{F}(k)
$$
and so if we take the Fourier transform of $\Lambda_\ep F$ we get
$
\hat{\Lambda_\ep F}(k)=\lambda(\ep k) \hat{F}(k)
$
where $\lambda$ is as in \eqref{lambda}.  This calculation demonstrates the central role that $\lambda$ plays in this business.
If we (formally)  expand $\lambda(\ep k)$ about $k=0$, use $c_0^2 = \lambda(0)$ and recall that $\lambda$ is  even we see that
$$
\hat{\B_\ep W_\ep} (k)= \ep^{-2} (c_0^2-{1\over2} \lambda''(0) \ep^2 - \lambda(\ep k)) \hat{W}_\ep = 
\left(-{1 \over 2} \lambda''(0) (1+k^2) + \O(\ep^2k^4)\right) \hat{W}_\ep(k).
$$

The usual identification of the operator $\partial_x^2$ with the Fourier multiplier $-k^2$ motivates setting
\be\label{B0}
\B_0 := -{1 \over 2}\lambda''(0)(1-\partial_x^2).
\ee
We will provide precise estimates on $\B_\ep - \B_0$ below. 
Moreover, in \cite{HML} it is shown that $\A_h$ converges strongly to the identity on $L^2$, which leads us to putting
\be\label{Q0}
\Q_0(V,W) := b VW \text{ where } b:=\left( \sum_{m\ge1} \beta_m m^3\right).
\ee
Again, precise estimates on $\Q_\ep - \Q_0$ are forthcoming. Similarly, the estimates for $\Psi'$ in \eqref{expando} indicate that $\P_\ep$ is formally  $\O(1)$ with respect to $\ep$.

Thus, if we put $\ep =0$ in \eqref{TWE2} in accordance with \eqref{B0} and \eqref{Q0}, we arrive at
\be\label{KdV}
-{1 \over 2} \lambda''(0)\left(W_0 - W_0''\right)  = bW_0^2.
\ee
This equation is well-known to have a unique even nontrivial homoclinic solution 
given by
\be\label{W0}
W_0(x) := -{3\lambda''(0) \over 4 b}\sech^2 \left( {x \over2}\right).
\ee
Obviously this is nonsense if $b = 0$ and that is why we assume it is not in Assumption \ref{bass}.
Note that $W_0$ is smooth and exponentially decaying and so, for all $s \in \R$,   $\|W_0\|_{H^{s}} < \infty$.

\begin{remark}\label{speed scale}
If, at \eqref{WSA}, we put $c_\ep^2 = c_0^2 + \mu \ep^2$ and repeat the above computation, at \eqref{KdV} we would
have the equation $\mu W_0 + {1 \over 2} \lambda''(0) W_0'' = b W_0^2$.  This equation will have a $\sech^2$-type
homoclinic only if $\mu \lambda''(0)<0$.  This is the main reason for the choice of the prefactor
at~\eqref{WSA}.
\end{remark}

The rest of the paper is dedicated to determining what happens to $W_0$ when $\ep$ is taken to be small and non-zero in \eqref{TWE2}.  
The keen-eyed reader will note that terms in $\B_\ep-\B_0$ contain high order derivatives which in turn raise the grim specter of singular perturbation.

\section{Core estimates}

\subsection{Estimates of $\B_\ep$} 
The linear operator $\B_\ep$ is the central operator in this work and in particular we need to understand its behavior as $\ep \to0^+$.  

Our main estimates are contained in the following:
\begin{proposition} \label{B ests}
Suppose that \eqref{motion} is Type I.
For all $s\in\R$ and $\ep>0$, $\B_\ep$ is a bounded and invertible map from $H^{s}$ to itself. Moreover, 
 there exists $C_\B>0$ and $\ep_\B>0$ so that the following  estimates hold for all $\ep\in (0,\ep_\B)$ and $s\in \R$:
\be\label{B est 1}
\|\B_\ep F \|_{H^s} \le C_\B\ep^{-2} \|F\|_{H^s},
\ee
\be\label{B inv est}
\|\B_\ep^{-1} F\|_{H^s} \le C_\B\|F\|_{H^s},
\ee
\be\label{B est 2}
\|(\B_\ep - \B_0) F\|_{H^s} \le C_\B \ep^\sigma \|F\|_{H^{s+\sigma}}
\ee
and
\be\label{B est 3}
\|(\B^{-1}_\ep - \B^{-1}_0) F\|_{H^s}  \le C_\B \ep^\sigma \|F\|_{H^s}.
\ee
\end{proposition}

\begin{remark}
The most important estimate in the above is \eqref{B est 3}, as it allows us to avoid problems related
the loss of derivatives associated with approximating $\B_\ep$ by $\B_0$, {\it i.e.}~to dispell issues of singular perturbation. 
We take this idea directly from the landmark FPUT paper \cite{FP1}.
\end{remark}

\begin{proof}


%

%
%
%
%

Recall that 
$\ds
\hat{\B_\ep F} (k) = \ep^{-2}\left(c_0^2-{1 \over 2}\lambda''(0)\ep^2 - \lambda(\ep k)\right) \hat{F}(k).
$
Let $\ulambda:=\inf_{k \in \R} \lambda(k)>-\infty$. 
Type I conditions  tell us $\ulambda \le \lambda(k) \le  \lambda(0)=c_0^2$ and $\lambda''(0)<0$. Thus the multiplier for $\B_\ep$ satisfies
 $$
 {1 \over 2} |\lambda''(0)| \le \ep^{-2}\left(c_0^2-{1 \over 2}\lambda''(0)\ep^2 - \lambda(\ep k)\right) \le (c_0^2-\ulambda) \ep^{-2} + {1 \over 2} |\lambda''(0)|. 
$$
And so the usual tools for estimating Fourier multiplier operators\footnote{By which we mean the following: if $\hat{LF}(k) = \tilde{L}(k) \hat{F}(k)$
and $C_L:=\sup_{k \in \R} |\tilde{L}(k)| (1+|k|)^{-q} < \infty$ then $L$ is a bounded map from $H^{s+q}$ to $H^{s}$ and $\|L F\|_{H^s} \le C_L \|F\|_{H^{s+q}}$. 
} 
allow us to conclude that $\B_\ep$ is a bounded 
 map with bounded inverse from $H^s$ to $H^s$, any $s\in \R$ and the estimates \eqref{B est 1} and \eqref{B inv est} follow.

Let
$
T_2(k) := \lambda(k) - \lambda(0) - {1 \over 2} \lambda''(0) k^2.
$
Type I condition (iii) tells us that $|T_2(k)| \le \mu_* |k|^{2+\sigma}$ when $|k| \le k_*$.
On the other hand, for all $k$ we have
$$
|T_2(k)| \le |\lambda(k) - \lambda(0)| + {1 \over 2} |\lambda''(0)| k^2 \le \lambda(0)-\ulambda + {1 \over 2} |\lambda''(0)| k^2.
$$
From this we can conclude that there is $\tilde{\mu} > 0$ so that
$|T_2(k)| \le \tilde{\mu} k^2$ for $|k| \ge k_*$. We can make this estimate worse by replacing $k^2$ 
with $|k|^{2+\sigma}$ by tinkering with the coefficient. Indeed, if we do that we find that there exists
$\mu > 0$ so that
\be\label{T2 est}
|T_2(k)| \le \mu |k|^{2+\sigma}
\ee
holds for all $k \in \R$.

Now we have
$$
\F[(\B_\ep - \B_0) F](k) = \left( {\ep^{-2}(\lambda(0) - {1 \over 2} \lambda''(0) \ep^2 - \lambda(\ep k)) }+
{{1 \over 2}  \lambda''(0) (1+k^2)}\right) \hat{F}(k).
$$
A direction computation followed by the estimate \eqref{T2 est} for $T_2$ gives:
\bes
\left\vert
{\ep^{-2}(\lambda(0) - {1 \over 2} \lambda''(0) \ep^2 - \lambda(\ep k)) }+
{{1 \over 2}  \lambda''(0) (1+k^2)}
\right\vert =  \ep^{-2} |T_2(\ep k)| \le \mu \ep^{\sigma}|k|^{2 + \sigma}.
\ees
This implies \eqref{B est 2}.

The key
to establishing \eqref{B est 3} is to show that the symbol for $\B_{\ep}^{-1}$ converges
uniformly on $\R$ to that of $\B_0^{-1}$.  This strategy first appeared for classical FPUT in \cite{FP1} and similar estimates have appeared in other contexts since \cite{HML,IP2,SW}.
The approach we take is closest to that of \cite{SW}, though  there are some technical differences.

We have:
$$
\F[(\B^{-1}_\ep - \B^{-1}_0) F](k) = \left({\ep ^2 \over \lambda(0) - {1 \over 2} \lambda''(0) \ep^2 - 
\lambda(\ep k)}
+{2 \over   \lambda''(0)  \left( 1+
k^2\right)}\right) \hat{F}(k).
$$
Some algebra leads us to
\bes\begin{split}
\left|{\ep ^2 \over \lambda(0) - {1 \over 2} \lambda''(0) \ep^2 - 
\lambda(\ep k)}
+{2 \over   \lambda''(0)  \left( 1+
k^2\right)}\right|
=& { 2 \over |\lambda''(0)|} { |T_2(\ep k)| \over 
| -{1 \over 2} \lambda''(0) \ep^2 -
T_1(\ep k)| \left( 1+
k^2\right)}
\end{split}\ees
where $T_1(k) := \lambda(k) - \lambda(0)$.
Type I condition (iii)  implies that $-{1 \over 2} \lambda''(0) \ep^2 -
T_1(\ep k) \ge -{1 \over 2} \lambda''(0) \ep^2 + \mu_* \ep^2 k^2 > 0$ when $|k| \le k_*$.
This and \eqref{T2 est} give
\be\begin{split}
&\sup_{|k|\le k_*/\ep}
\left|{\ep ^2 \over \lambda(0) - {1 \over 2} \lambda''(0) \ep^2 - 
\lambda(\ep k)}
+{2 \over   \lambda''(0)  \left( 1+
k^2\right)}\right| \\
\le &C\sup_{|k|\le k_*/\ep}
 { |\ep k|^{\sigma+2} \over 
(  -{1 \over 2} \lambda''(0)\ep^2  + \mu_*\ep^2  k^2)(1+k^2)} \\
\le & C \ep^\sigma \sup_{|k|\le k_*/\ep}  {|k|^{\sigma+2} \over 
(  -{1 \over 2} \lambda''(0)  + \mu_*  k^2)(1+k^2)}\\
\le& C \ep^\sigma.
\end{split}\ee

On the other hand, we have
\be\label{win}\begin{split}&\sup_{|k|\ge k_*/\ep}
\left|{\ep ^2 \over \lambda(0) - {1 \over 2} \lambda''(0) \ep^2 - 
\lambda(\ep k)}
+{2 \over   \lambda''(0)  \left( 1+
k^2\right)}\right| \\
\le &
\sup_{|k|\ge k_*/\ep}
\left|{\ep ^2 \over \lambda(0) - {1 \over 2} \lambda''(0) \ep^2 - 
\lambda(\ep k)}\right| 
+\sup_{|k|\ge k_*/\ep}\left|{2 \over   \lambda''(0)  \left( 1+
k^2\right)}\right| 
\end{split}\ee
Conditions (ii) and (iv) imply that, for any $\ep > 0$,
$$
\inf_{|k| \ge k_*/\ep}  \left(\lambda(0) - {1 \over 2} \lambda''(0) \ep^2-\lambda(\ep k)\right)
\ge \lambda(0) - \sup_{|k| \ge k_*} \lambda(k)> 0.
$$
That is to say, the denominator in the first term on the right is bounded from zero uniformly.
Thus the first term is
controlled by $C \ep^2$.  And so is the second, by yet more elementary considerations.  

We conclude 
\be\begin{split}
&\sup_{|k| \in \R}
\left|{\ep ^2 \over \lambda(0) - {1 \over 2} \lambda''(0) \ep^2 - 
\lambda(\ep k)}
+{2 \over   \lambda''(0)  \left( 1+
k^2\right)}\right|  \le C \ep^\sigma.
\end{split}\ee
This implies \eqref{B est 3}, and we are done.
%

\end{proof}

\subsection{Estimates of $\A_h$}
The averaging operators $\A_h$ play an important role and we record some key estimates:
\begin{theorem}
\label{hml}
There exists $C_\A>0$ so that the following hold for all $h > 0$, $s\in \R$:
 \be \label{A est 1}
 \ds \|\A_h F\|_{H^s} \le \|F\|_{H^s},
 \ee
\item \be \label{A est smooth}  \|\A_h F\|_{H^{s+1}} \le C_\A (1+h^{-1}) \|F\|_{H^s},\ee
\be \label{A est 2}
\ds \|\A_hF-F\|_{H^s} \le C_\A h^2 \|F''\|_{H^s},
\ee
\be\label{A est 3}
 \left\|\A_hF-F-{h^2 \over 24} F''\right\|_{H^s}  \le C_\A h^4 \|F''''\|_{H^s}
 \ee
 and
\be \label{A est 4}
\ds \left\| (\A_h -\A_{h'}) F\right\|_{H^s} \le C_\A |h-h'| \|F'\|_{H^s}.
\ee

\end{theorem}


\begin{proof}
For  $H^0=L^2$ the details can be found in \cite{HML}, but the case for general $H^s$ is more or less no different and so we largely omit
the details.  We do provide a slightly different take on the proof of \eqref{A est 2} here, as it will have a byproduct which will be useful below.

We have $\ds
\F [ \A_h F - F](k) = (\sinc(hk/2)-1) \hat{F}(k).$ 
The function $\sinc(K) - 1$ has a zero of order two at $K=0$ and this, with some Calculus trickery, 
implies that there is a constant $C_\Theta>0$ so that $\sup_{K\in \R} (\sinc(K/2)-1)/K^2 < C_\Theta$. So if we define the Fourier multiplier operator $\Theta_h F$ via 
\be\label{Theta}
\ds \hat{\Theta_h F} (k) =  {\sinc(h k/2)-1 \over (h k)^2} \hat{F}(k),
\ee
we have, for all $h > 0$, 
\be \label{Theta est}
\|\Theta_h F\|_{H^s} \le C_\Theta \|F\|_{H^{s}}.
\ee
Some algebra on the Fourier side shows us that $(\A_h-1) F = -h^2 \Theta_h F''$.  This observation and \eqref{Theta est} give \eqref{A est 2}.

\end{proof}

\subsection{Estimates of $\Q_\ep$}
For the bilinear term $\Q_\ep$ we have:
\begin{proposition} \label{Qest}
Suppose Assumption \ref{bass} holds. Then, for $s\ge1$ 
there exists $C_\Q>0$ and $\ep_\Q > 0$ so that the following estimates hold for $\ep,\ep' \in (0,\ep_\Q)$:
\be\label{Q est 1}
\|\Q_\ep(V,W)\|_{H^s} \le C_\Q \|V\|_{H^s} \|W\|_{H^s},
\ee
\be\label{Q est smooth}
\|\Q_\ep(V,W)\|_{H^{s+1}} \le C_\Q \ep^{-1} \|V\|_{H^s} \|W\|_{H^s},
\ee
\be\label{Q est 2}
\|\Q_\ep(V,W) - \Q_{\ep'}(V,W)\|_{H^s} \le C_\Q|\ep-\ep'|\|V\|_{H^{s+1}} \|W\|_{H^{s+1}},
\ee
\be\label{Q est 3}
\|\Q_\ep(V,W) - \Q_0(V,W)\|_{H^s} \le C_\Q \ep^2\|V\|_{H^{s+2}} \|W\|_{H^{s+2}}
\ee
and
\be\label{Q est 4}
\| (1-\partial_x^2)^{-1} \left(Q_\ep(V,W)-Q_0(V,W)\right) \|_{H^1} \le C_\Q \ep^2 \| V\|_{H^1} \|W\|_{H^1}.
\ee

\end{proposition}

\begin{proof}
For \eqref{Q est 1} we use \eqref{A est 1} and Sobolev embedding\footnote{By which we mean the famous estimate $\|f\|_{L^\infty} \le \|f\|_{H^s}$ and its best friend $\| f g \|_{H^s} \le C\|f\|_{H^s} \|g\|_{H^s}$, when $s \ge 1$.} to see that
$$\| \A_{\ep m} [(\A_{\ep m}V)  (\A_{\ep m} W)]\|_{H^s} \le C\|V\|_{H^s} \|W\|_{H^s}$$
where the constant is independent of both $m$ and $\ep$.  Therefore
$$
\|\Q_{\ep}(V,W)\|_{H^s} \le C \left(\sum_{m \ge 1} |\beta_m| m^3\right) \|V\|_{H^s} \|W\|_{H^s}.
$$
Assumption \ref{bass}
tells us that the sum converges  and so we get \eqref{Q est 1}. The estimate \eqref{Q est smooth} 
follows in a similar fashion, simply using \eqref{A est smooth} instead of \eqref{A est 1} in the first step on the outer $\A_{\ep m}$.

For \eqref{Q est 2} we have
\bes\begin{split}
\A_{\ep m} [(\A_{\ep m}V)  (\A_{\ep m} W)] - \A_{\ep' m} [(\A_{\ep' m}V)  (\A_{\ep' m} W)] 
&= \left(\A_{\ep m} -\A_{\ep'm}\right)[(\A_{\ep m}V)  (\A_{\ep m} W)]\\
&+ \A_{\ep' m}[((\A_{\ep m}-\A_{\ep'm})V)  (\A_{\ep m} W)]\\
&+ \A_{\ep' m}[(\A_{\ep'm}V) \left((\A_{\ep m} -\A_{\ep'm})W\right)].
\end{split}\ees
Estimates \eqref{A est 1} and \eqref{A est 4} together with Sobolev embedding give
$$
\|\A_{\ep m} [(\A_{\ep m}V)  (\A_{\ep m} W)] - \A_{\ep' m} [(\A_{\ep' m}V)  (\A_{\ep' m} W)] \|_{H^s} \le C m|\ep - \ep'|  
 \|V\|_{H^{s+1}} \|W\|_{H^{s+1}}.
$$
Thus
$$
\|\Q_\ep(V,W) - \Q_{\ep'}(V,W)\|_{H^s}  \le C|\ep-\ep'|\left(\sum_{m \ge 1} |\beta_m|  m^4\right)\|V\|_{H^{s+1}} \|W\|_{H^{s+1}}.
$$
As before, Assumption \ref{bass}
tells us that the sum converges and  \eqref{Q est 2} follows.

For  \eqref{Q est 3} we have
\be\label{add sub}\begin{split}
\A_{\ep m} [(\A_{\ep m}V)  (\A_{\ep m} W)] - VW 
&= \left(\A_{\ep m} -1\right)[(\A_{\ep m}V)  (\A_{\ep m} W)]\\
&+ [(\A_{\ep m}-1)V]  \A_{\ep m} W\\
&+ V \left(\A_{\ep m} -1\right) W.
\end{split}\ee
Estimates \eqref{A est 1}, \eqref{A est 2} and Sobolev embedding give
$$
\|\A_{\ep m} [(\A_{\ep m}V)  (\A_{\ep m} W)] - VW \|_{H^s} 
\le C \ep^2 m^2 \|V\|_{H^{s+2}} \|W\|_{H^{s+2}}.
$$
And so
$$
\|\Q_\ep(V,W) - \Q_0(V,W)\|_{H^s} \le C \ep^2 \left(\sum_{m \ge 1} |\beta_m| m^5 \right) \|V\|_{H^{s+2}} \|W\|_{H^{s+2}}.
$$
The sum converges and thus we have \eqref{Q est 3}.

Things are a bit trickier for \eqref{Q est 4}.  Let 
\begin{multline*}
I_h:=(1-\partial_x^2)^{-1}\left(\A_{h} -1\right)[(\A_{h}V)  (\A_{h} W)],\\
II_h:= (1-\partial_x^2)^{-1}[(\A_{h}-1)V]  \A_{h} W \mand 
III_h:=(1-\partial_x^2)^{-1}V \left(\A_{h} -1\right) W
\end{multline*}
so that
$$
(1-\partial_x^2)^{-1} \Q_\ep(V,W) = \sum_{m \ge 0} \beta_m m^3\left(I_{\ep m} + II_{\ep m} + III_{\ep m}\right).
$$

We estimate $III_h$ in $H^1$.
Integration by parts shows that 
$$
\int_\R (1-\partial_x^2)III_h(x) III_h(x) dx = \int_\R III_h^2(x) + (\partial_x III_h)^2(x) dx = \|III_h\|_{H^1}^2.
$$
And so, using the definition of $III_h$, we have:
$$
 \|III_h\|_{H^1}^2 = \int_\R III_h(x) V(x)\left(\A_{h} -1\right) W(x) dx.
$$
Recalling the fact that $(\A_h -1) = -h^2 \Theta_h \partial_x^2$ (see \eqref{Theta}) we have
$$
 \|III_h\|_{H^1}^2 = -h^2 \int_\R III_h(x) V(x)\Theta_h  W''(x) dx.
$$
Integrating by parts (and noting that $\Theta_h$ and the derivative commute) we get:
$$
 \|III_h\|_{H^1}^2 = h^2 \left(\int_\R III'_h(x) V(x)\Theta_h  W'(x) dx+ \int_\R III_h(x) V'(x)\Theta_h  W'(x) dx\right).
$$
From this we use various versions of H\"older's inequality to get:
$$
 \|III_h\|_{H^1}^2 \le  h^2 \left(\|III'_h\|_{L^2} \|V\|_{L^\infty} \|\Theta_h W'\|_{L^2} +
 \|III_h\|_{L^\infty} \|V'\|_{L^2} \|\Theta_h W'\|_{L^2}
 \right).
$$
Using \eqref{Theta est} and Sobolev's inequality and we get
$\ds
 \|III_h\|_{H^1}^2 \le  C h^2 \|III_h\|_{H^1} \|V\|_{H^1} \|W\|_{H^1}
$ or rather
$$
 \|III_h\|_{H^1} \le  C h^2  \|V\|_{H^1} \|W\|_{H^1}.
$$

We can use the same sort of reasoning to show that $
 \|I_h\|_{H^1} +  \|II_h\|_{H^1} \le  C h^2  \|V\|_{H^1} \|W\|_{H^1}.
$ as well. Thus we have
$$
\|(1-\partial_x^2)^{-1} \Q_\ep(V,W) \|_{H^1} \le C \ep^2 \left(\sum_{m \ge 1} |\beta_m| m^5\right) \|V\|_{H^1} \|W\|_{H^1}.
$$
The sum converges, \eqref{Q est 4} follows and we are done.

\end{proof}

\subsection{Estimates of $\P_\ep$}
Our estimates for  the ``cubic'' part $\P_\ep$ are:
\begin{proposition}
Assume Assumption \ref{phass} and \ref{bass}.  Then
there exists 
$C_\P>0$ so that for all $\kappa_* > 0$ there exists $\ep_\P>0$ so that 
so that the following hold when $\|W\|_{H^1},\|\tilde{W}\|_{H^1} \le \kappa_*$ and $\ep \in (0,\ep_\P)$:
\be\label{P est 1}
\|\P_\ep(W)\|_{H^1} \le C_\P \|W\|_{H^1}^3
\ee
and
\be\label{P est 2}
\|\P_\ep(W)-\P_\ep(\tilde{W})\|_{H^1} \le C_\P\left(\|W\|^2_{H^1} + \|\tilde{W}\|_{H^1}^2 \right) \|W-\tilde{W}\|_{H^1}.
\ee
Moreover, if $s = 1,2$,
\be\label{P est smooth}
\|\P_\ep(W)\|_{H^{s+1}} \le C_\P \ep^{-1} \|W\|_{H^s}^3.
\ee

\end{proposition}

\begin{proof}
Estimate \eqref{P est 1} is a consequence of  \eqref{P est 2}, so we prove \eqref{P est 2}.

The first step is to notice if
$|a|,|b| \le m \delta_*$ (as in Assumption \ref{phass}) then 
the estimates in \eqref{Psi est} give
\be\label{diff of cube}
|\Psi_m'(a)-\Psi_m'(b)| \le 
{3 \over 2} {\gamma_m }\left(a^2 + b^2\right)|a-b|
\mand
|\Psi_m''(a)-\Psi_m''(b)| \le 
6 {\gamma_m }\left(|a| + |b|\right)|a-b|.
\ee

Now fix $\kappa_*>0$ and let $\ep_\P = \sqrt{\delta_*/\kappa_*}$.
Then 
$\|W\|_{H^1},\|\tilde{W}\|_{H^1} \le \kappa_*$ and $\ep \in (0,\ep_\P)$ imply (by way of \eqref{A est 1} and Sobolev embedding) that
$\|m \ep^2 \A_{\ep m} W\|_{L^\infty},\|, m \ep^2 \A_{\ep m} \tilde{W}\|_{L^\infty} \le m \delta_*$.  
So we can deploy the first estimate in \eqref{diff of cube} 
to get
\bes\begin{split}
&\left \vert 
\Psi'_m(m \ep^2 \A_{\ep m} W(x))
-\Psi'_m(m \ep^2 \A_{\ep m} \tilde{W}(x))
\right \vert\\ \le &{3\gamma_m m^3 \ep^6 \over 2} \left(| \A_{\ep m} W(x)|^2 + |\A_{\ep m} \tilde{W}(x)|^2\right)
\left|\A_{\ep m} W(x) - \A_{\ep m} \tilde{W}(x)\right|.
\end{split}\ees
This implies, using Sobolev embedding and estimates from Theorem \ref{hml}:
$$
\| \Psi'_m(m \ep^2 \A_{\ep m} W) - \Psi_m'(m \ep^2 \A_{\ep m}\tilde{W})\|_{L^2}
\le {3 \over 2} \gamma_m m^3 \ep^6 \left(\|W\|^2_{H^1} + \|\tilde{W}\|_{H^1}^2 \right) \| W- \tilde{W}\|_{L^2}.
$$

Similarly we have
\be\begin{split}
&\left \vert 
\partial_x \left(\Psi'_m(m \ep^2 \A_{\ep m} W(x))
-\Psi'_m(m \ep^2 \A_{\ep m} \tilde{W}(x))\right)
\right \vert\\
  =&m \ep^2 \left \vert 
\Psi''_m(m \ep^2 \A_{\ep m} W(x)) \A_{\ep m} W'(x)
-\Psi''_m(m \ep^2 \A_{\ep m} \tilde{W}(x)) \A_{\ep m} \tilde{W}'(x)
\right \vert\\
 \le & m \ep^2 \left \vert \left(\Psi''_m(m \ep^2 \A_{\ep m} W(x)) 
-\Psi''_m(m \ep^2 \A_{\ep m} \tilde{W}(x))\right) \A_{\ep m} {W}'(x)
\right \vert\\
&+ m \ep^2 \left \vert\Psi''_m(m \ep^2 \A_{\ep m} \tilde{W}(x))\A_{\ep m}
\left(W'(x)-\tilde{W}'(x)   \right)
\right \vert.
\end{split}\ee
Using this, \eqref{A est 2}, the second estimate in \eqref{diff of cube}
and Sobolev we get
$$
\left \|
\partial_x \left(\Psi'_m(m \ep^2 \A_{\ep m} W)
-\Psi'_m(m \ep^2 \A_{\ep m} \tilde{W})\right)
\right\|_{L^2} \le 9 \gamma_m m^3 \ep^6 \left(\|W\|^2_{H^1} + \|\tilde{W}\|_{H^1}^2 \right) \| W- \tilde{W}\|_{H^1}.
$$
Thus we have
$$
\|\P_\ep(W)-\P_\ep(\tilde{W})\|_{H^1} 
\le C \left(\sum_{m \ge 1} \gamma_m m^4\right) \left(\|W\|^2_{H^1} + \|\tilde{W}\|_{H^1}^2 \right) \| W- \tilde{W}\|_{H^1}.
$$
The sum converges because of Assumption \ref{bass} and we get the estimate from there.

To prove \eqref{P est smooth} is more of the same, simply using \eqref{A est smooth} 
to estimate the outermost instance of $\A_{\ep m}$ in the definition of $\P_\ep$.  We omit the details.

\end{proof}

\section{Solitary waves}

Let $\lambda(k)$ be Type I and take Assumptions \ref{phass} and \ref{bass} 
as given.
We prove the existence of solitary wave solutions, which is to say a nontrivial solution of \eqref{TWE1} in $H^1$.
We begin at \eqref{TWE2} and put
$$
W_\ep = W_0 + \ep^{\sigma} V_\ep.
$$
Routine computations show that $V_\ep$ solves
\be\label{TWESOL}\begin{split}
\B_\ep V_\ep - 2 \Q_\ep(W_0,V_\ep)  
&=R_\ep +\ep^\sigma \Q_\ep(V_\ep,V_\ep) + \ep^2 \N_\ep (V_\ep)
\end{split}
\ee
with 
$$
R_\ep := \ep^{-\sigma} \left[-\B_\ep W_0 + \Q_\ep(W_0,W_0) +\ep^2 \P_\ep(W_0) \right]
$$
and
$$
\N_\ep (V): =  \ep^{-\sigma}\left[\P_\ep ( W_0 + \ep^\sigma V) - \P_\ep(W_0) \right].
$$
From Proposition \ref{B ests} we know that $\B_\ep$ is invertible and so 
\eqref{TWESOL} is equivalent to
\be\label{TWES2}
\L_\ep V:= V_\ep - 2\B_{\ep}^{-1} \Q_\ep(W_0,V_\ep) = \B_\ep^{-1} R_\ep + \ep^\sigma \B_\ep^{-1} \Q_\ep(V_\ep,V_\ep) + \ep^2 \B_\ep^{-1} \N_\ep(V_\ep).
\ee

Now we claim that 
\be\label{claim}
\|\B_\ep^{-1} \Q_\ep(W_0,V) - \B_0^{-1} \Q_0(W_0,V)\|_{H^1}  \le C \ep^\sigma \|V\|_{H^1}.
\ee
Here is why. We have
$$
\B_\ep^{-1} \Q_\ep(W_0,V) - \B_0^{-1} \Q_0(W_0,V)
= \left(\B_\ep^{-1} - \B_0^{-1}\right) \Q_\ep(W_0,V) + \B_0^{-1} \left(\Q_\ep(W_0,V) - \Q_0(W_0,V) \right).
$$
For the first term we use \eqref{B est 3} and \eqref{Q est 1} to get
$$
\| \left(\B_\ep^{-1} - \B_0^{-1}\right) \Q_\ep(W_0,V)\|_{H^1} \le C_\B C_\Q \ep^\sigma \|W_0\|_{H^1} \|V\|_{H^1}.
$$
For the second we recall the definition of $\B_0$ and then use \eqref{Q est 4}
$$
\|\B_0^{-1} \left(\Q_\ep(W_0,V) - \Q_0(W_0,V) \right)\|_{H^1} \le {|\lambda''(0)| \over 2} C_\Q \ep^2 \|W_0\|_{H^1} \|V\|_{H^1}.
$$
Thus we have \eqref{claim}.

And so we see that 
$\L_\ep V$
is a small perturbation (in the norm topology of bounded operators from $H^1$ to $H^1$) of 
$$
\L_0 V:=V - 2 \B_0^{-1} \Q_0(W_0,V) = V + {4b \over \lambda''(0)}(1-\partial_x^2)^{-1} [W_0 V].
$$
$\L_0$  is invertible on $E^1$ (recall $E^1=H^1 \cap\left\{\text{even functions}\right\}$).
See, for instance, Proposition 4.1 in \cite{FP1}, Lemma 4 in \cite{SW} or Lemma 3.1 in \cite{HML}.
And thus a Neumann series argument implies $\L_\ep$ is also invertible on $E^1$
and there is constant $C_\L>0$ so that 
\be\label{Lep inv est}
\|\L_\ep^{-1} F\|_{E^1} \le C_\L \|F\|_{E^1}
\ee
and this holds for $\ep$ sufficiently close to zero.

If we impose the condition that $V_\ep$ is even, it is simple enough to conclude that everything on the right hand side of \eqref{TWES2} is even
and so we invert $\L_\ep$ to get
\be\label{TWES3}
V_\ep =\L_\ep^{-1} \B_\ep^{-1} R_\ep + \ep^\sigma\L_\ep^{-1} \B_\ep^{-1} \Q_\ep(V_\ep,V_\ep) + \ep^2\L_\ep^{-1} \B_\ep^{-1} \N_\ep(V_\ep)
=:\M_\ep[V_\ep].
\ee
We now show that $\M_\ep$ is a contraction on a ball in $E^1$, which in turns means we will have a fixed point and thus a solution of \eqref{TWES3} (and thus of our whole problem).

First we note that if we use \eqref{KdV} and the definition $R_\ep$ we have:
$$
R_\ep = -\ep^{-\sigma}
(\B_\ep - \B_0) W_0
+\ep^{-\sigma}\left(Q_\ep(W_0,W_0)-Q_0(W_0,W_0)\right) - \ep^{2-\sigma} \P_\ep(W_0).
$$
Then we use \eqref{B est 2} on the first term, \eqref{Q est 3} on the second and \eqref{P est 1} on the third to get
$$
\|R_\ep\|_{H^1} \le C_\B \|W_0\|_{H^{1 + \sigma}} + C_\Q \ep^{2-\sigma} \|W_0\|^2_{H^3} + C_\P \ep^{2-\sigma}\|W_0\|_{H^1}^3.
$$
Recalling that $\|W_0\|_{H^s} < \infty$ for all $s$ and then using \eqref{Lep inv est} and \eqref{B inv est} we conclude that there is are
 constants
$\kappa_1,\ep_1>0$ so that
$$
\|\L_\ep^{-1} \B_\ep^{-1} R_\ep\|_{H^1} \le \kappa_1/2
$$
when $\ep \in (0,\ep_1)$.

Likewise if we use \eqref{B inv est}, \eqref{Q est 1} and \eqref{Lep inv est}  we can find $\kappa_2, \ep_2>0$ so that
$$
\|\ep^\sigma\L_\ep^{-1} \B_\ep^{-1} \Q_\ep(V,V) \|_{H^1} \le \ep^\sigma \kappa_2 \|V\|_{H^1}^2 
$$
and 
\bes\begin{split}
\|\ep^\sigma\L_\ep^{-1} \B_\ep^{-1}\left( \Q_\ep(V,V) - \Q_\ep(\tilde{V},\tilde{V})\right) \|_{H^1} 
\le \ep^\sigma  \kappa_2\left(\|V\|_{H^1} + \|\tilde{V}\|_{H^1} \right)\|V-\tilde{V}\|_{H^1}
\end{split}\ees
when $\ep \in (0,\ep_2)$.

Now assume that
$\|V\|_{H^1}, \|\tilde{V}\|_{H^1} \le \kappa_1$.  Then we can use 
 \eqref{B inv est}, \eqref{Lep inv est}, the definition of $\N_\ep$ and \eqref{P est 2} to show 
 there exists $\kappa_3,\ep_3>0$ so that 
$$
\|\ep^2 \L_\ep^{-1} \B_\ep^{-1} \N_\ep(V)\|_{H^1} 
\le  \ep^{2}\kappa_3\|V\|_{H^1}
$$
and
$$
\|\ep^2 \L_\ep^{-1} \B_\ep^{-1} \left(\N_\ep(V)-\N_\ep(\tilde{V}) \right)\|_{H^1} 
\le  \ep^{2} \kappa_3\|V-\tilde{V}\|_{H^1}
$$
when $\ep \in (0,\ep_3)$.

All the preceding estimates tell us that if $0 < \ep < \ep_* := \min\{\ep_1,\ep_2,\ep_3\}$ 
and $\|V\|_{H^1}, \|\tilde{V}\|_{H^1} \le \kappa_1$
then
$$
\|\M_1 [V]\|_{H^1} \le {1 \over 2} \kappa_1 +  \ep^\sigma \kappa_2 \kappa^2_1 + \ep^2 \kappa_3 \kappa_1
$$
and
$$
\|\M_1 [V]-\M_1[\tilde{V}]\|_{H^1}  \le (2\ep^\sigma \kappa_2 \kappa_1 + \ep^2 \kappa_3) \|V-\tilde{V}\|_{H^1}.
$$

Thus we can take $\ep$ sufficient small so that $\M_1$ maps the ball of radius $\kappa_1$ in $E^1$ into itself and is a contraction there.
The fixed point of the map solves \eqref{TWE1} and thus results in a traveling wave solution of \eqref{motion}.
Moreover, the smoothing estimates  \eqref{Q est smooth} and \eqref{P est smooth} can be used in a bootstrap
argument to show that $V_\ep$ is in fact in $H^3$; the details are routine and we leave them out.
All together we have proven our main result:
\begin{theorem} \label{type I sw}
Suppose that $\lambda(k)$ is Type I and Assumptions \ref{phass} and \ref{bass} hold.
There exists $\ep_I>0$, and $\kappa_I>0$
so the following hold for $\ep \in (0, \ep_1)$.  There is a unique function $V_\ep \in E^3$ with $\|V_\ep\|_{E^1} \le \kappa_I$
so that
$$
c_\ep^2 = c_0^2 - {1 \over 2} \lambda''(0) \ep^2 \mand W(x) = W_0(x) + \ep^\sigma V_\ep(x)
$$
solve \eqref{TWE1}. 
\end{theorem}

\section{Examples of Type I lattices}

\subsection{Hermann/Mikikits-Leitner FR lattices}\label{mh section}
Theorem \ref{type I sw}  recaptures the bulk of the main result for the finite range problem as studied 
 in \cite{HML} (their Corollary 14).  In addition to an assumption on the smoothness of the potentials $\Phi_m$
 similar to our Assumption \ref{phass}, they have three major conditions:
 \begin{enumerate}[(a)]
 \item
 $\al_m, \beta_m$ and $\gamma_m$ are zero except
 at finitely many choices for $m$, \item the coefficients $\al_m$ are always non-negative and at least one is positive,
 \item $\sum_{m} \beta_m m^3 \ne 0$.
 \end{enumerate}

 Condition (a) immediately implies the convergences in Assumption \ref{bass} and (c) is exactly $b \ne 0$. The conditions also imply their lattice is Type I.  First of all,
condition (b) tells us that $\al_m \ge 0$ for all $m$ and we know immediately that $\lambda(k) \ge0$ for all $k$ and so 
we have condition (i).  Likewise, since $d^2/dk^2|_{k = 0} \sinc^2(k)<0$, condition (b) also  tells us that $\lambda''(0)<0$ and so we have (ii). 

Next, since $\sinc(k)$ is $C^{\infty}$ on $\R$,
condition (a) tells us that $\lambda(k)$ is $C^{\infty}$ on $\R$, as it is just a finite sum of $\sinc$s.  
The two estimates in (iii) are then just easy consequences of Taylor's theorem.  The second holds with
$\sigma = 2$.  A small note here is that $\lambda'(0) = \lambda'''(0) = 0$ because $\lambda(k)$ is even.

As for (iv), look at 
$
\lambda(0) - \lambda(k) = \sum_{m \ge 1} \al_m (1- \sinc^2(mk/2)).
$
If $k \ne 0$ then $1-\sinc^2(mk/2) > 0$ for all $m \ge 1$. Thus, since all the terms in the last sum are positive, the sum itself is positive for all $k$.
And we have $\lambda(0) > \lambda(k)$ if $k \ne 0$. And since $\lambda(k) \to 0$ as $|k| \to \infty$ and is continuous, we can conclude $\sup_{|k| \ge k_*} \lambda(k) < \lambda(0)$ for any non-zero $k_*$.

So the conditions they impose on their lattices easily fulfill all the assumptions needed in Theorem \ref{type I sw} and our result applies with $\sigma = 2$.
Their condition (b) does allow them to prove that their solitary wave is positive and unimodal. 

\subsection{NNN Lattices}
We now consider next nearest neighbor lattices as presented in \cite{W} or \cite{VZ}.  Those problems (after some elementary 
changes of variables and renaming constants) correspond to setting 
\bes\label{NNN}
\Phi'_1(r) = r + \beta_1 r^2 +\Psi_1'(r), \quad \Phi'_2(r) = g r  + \beta_2 r^2 + \Psi_2'(r) \mand \Phi'_m(r) = 0 \text{ when $m > 2$}.
\ees
The functions $\Psi'_1$ and $\Psi'_2$ are assumed smooth and enjoy the cubic type estimates in \eqref{Psi est}.  That is to say, their lattices pass Assumption \ref{phass} (with $r_* =0$). 
The constant  $g$ is what we call $\al_2$; $g$ is the name in both \cite{W} and \cite{VZ} and so we use it here for consistency.  The convergences in Assumption \ref{bass} are met because all the coefficients are zero after $m = 2$.
 We need $b \ne 0$ to hold, which here means that $\beta_1\ne -8 \beta_2$. In \cite{W}, they specify $\beta_1 \ne 0$ (which they call ``$a$'') and $\beta_2 =0$, so our condition encompasses theirs. In \cite{VZ} they require $0 < \beta_2 < \beta_1/2$ (their equation (7), where they use $\al$ where we use $\beta$).  Again,
 our condition encompasses theirs.

This lattice is Type I when $g>-1/16$.  Here is a quick explanation.  We have
$$
\lambda(k) = \sinc^2(k/2) + 4 g \sinc^2(k) \mand c_0^2 = 1+4g.
$$
This is clearly bounded below and we have (i).
Then we compute $\ds \lambda''(0) = -{1 \over 6} - {8 g \over 3}$ which tells us that  
$\lambda''(0)<0$ 
 when $g > - 1/16$.  As in the previous section, the smoothness of $\lambda(k)$ 
 and Taylor's theorem gives the estimates in (iii) with $\sigma = 2$.
 More differential Calculus can be used get condition (iv) for $g > -1/16$. 

 And so we can conclude the existence of solitary waves in NNN lattices as in \eqref{NNN} so long as $g > -1/16$
 and $\beta_1\ne -8 \beta_2$.  In particular we have the results of Theorem \ref{type I sw} with $\sigma = 2$.
   This is what is found in \cite{W}, though the calculations there are not fully rigorous.  In \cite{VZ} the authors study the case when $g \in (-1/4,-1/16)$.  It turns out that this case is an example of what we call Type II and that's a story for another time.

\subsection{Calogero-Moser} Now we show that certain generalized Calogero-Moser lattices (as studied in \cite{IP1,IP2}) meet the assumptions 
of Theorem \ref{type I sw} and thus establish the existence of solitary waves.  We recall that this lattice corresponds to putting $\Phi_m(r) =1 / r^a$
where $a > 1$ is a parameter.  Our particular interest is when $a > 3$.
In \eqref{TWA} we put, for simplicity, $r_* =1$.  

The first thing is to compute $\al_m$, $\beta_m$ and $\gamma_m$.
We have
$$
\al_m = \Phi''_m(m) = a(a+1)m^{-a-2} \mand 
\beta_m = \Phi'''_m(m)/2 = -{1 \over 2} {a(a+1)(a+2)} m^{-a-3}.
$$
It doesn't take too much effort to show that $\gamma_m =  a (a+1)(a+2)(a+3) m^{-a - 4} $ works in \eqref{Psi est}.
Which is to say we have Assumption \ref{phass} and the estimates in \eqref{Psi est}.  
Since $a > 3$  we see that the convergences in Assumption \ref{bass} are met as well.  And since $\beta_m < 0$ for all $m$ it follows that $b \ne 0$, as we want. 

So we need to confirm that the lattice is Type I.  First of all, since $\al_m >0$ for all $m$, we have $\lambda(k)\ge0$ and so (i) is easily confirmed.  
Indeed this also implies that $\lambda''(0) < 0$ so we have (ii).

Establishing (iii) is complicated as $\theta(k)$ is not a $C^\infty$ function. Indeed, its regularity depends
on $a$.  First note that $\theta(k) = \theta_a(k)$ where
$$
 \theta_a(k):=4 a(a+1) \sum_{m = 1}^\infty {1 \over m^{a+2}} \sin^2(mk/2)=
2a(a+1)\left( \sum_{m\ge1} {1 \over m^{a+2}} - \sum_{m \ge 1}{1 \over m^{a+2}} \cos(mk) \right).
$$
We claim that $\theta_a(k)$ is 
\begin{itemize}
\item $C^{4,a-3}$ on $\R$ when $a \in (3,4]$,
\item $C^{5,a-4}$ on $\R$ when $a \in (4,5)$,
\item $C^{5,1-\delta}$ on $\R$, for all $\delta >0$, when $a=5$,
\item $C^{5,1}$ on $\R$ when $a > 5$.
\end{itemize}
Most of this claim is a consequence of this following, 
which is Theorem 4.2 in \cite{N}:
\begin{theorem}\label{niss} Suppose that $f(k) = \sum_{m \in \Z} f_m e^{i mk}$ and, for some $r \in {\bf N}$ and $q \in(0,1)$, 
$\sup_{m \in \Z} |f_m| |m|^{r+q} < \infty$. Then $f \in C^{r-1,q}$ on $\R$.
\end{theorem}
All the statements in the claim come directly from this, except the case when $a = 4$.  But if you are persistent in doing lots of integrals you can find that
$$
\theta_4(k)= {2 \over 9} \pi^4 k^2 - {5 \over 18} \pi^2 k^4+{1 \over 6} \pi |k| k^4 -{1 \over 36} k^6
$$
on $k \in (-\pi,\pi]$ and is the $2\pi$-periodic extension of the above elsewhere.  And that function is easily checked to be $C^{4,1}$ on $\R$.
Note also that if $a > 5$, Theorem \ref{niss} implies that $\theta_a(k)$ is smoother than $C^{5,1}$.  That extra regularity does not translate into any particularly interesting extra features of the solution and so we simply lump all cases with $a>5$ into the one class.


Next put
$$
\eta_a(k) = \theta_a(k) - {1 \over 2} \theta_a''(0)k^2 - {1 \over 24} \theta_a''''(0) k^4.
$$
So then
$$
\lambda(k) = \lambda_a(k):={\eta_a(k) \over k^2} +  {1 \over 2} \theta_a''(0) + {1 \over 24} \theta_a''''(0) k^2
$$
and
$$
\lambda_a''(k) = {\eta_a''(k) \over k^2} - {4\eta_a'(k) \over k^2} + {6 \eta_a(k) \over k^4} + {1 \over 12} \theta_a''''(0).
$$

If $a \in (3,4]$ then $\theta_a(k) \in C^{4,a-3}$ and 
\be\label{eta int}
\eta_a(k) = \int_0^k \int_0^{k_1} \int_0^{k_2} \int_0^{k_3} \theta_a''''(k_4) dk_4 d k_3 dk_2 dk_1.
\ee
Since we know $\theta_a''''(k)$ is in $C^{0,a-3}$ we have
$$
|\eta_a(k)| \le C \left \vert 
\int_0^k \int_0^{k_1} \int_0^{k_2} \int_0^{k_3} |k_4|^{a-3} dk_4 d k_3 dk_2 dk_1
\right \vert \le C|k|^{a+1}.
$$
Similarly one can show (by differentiating the formula for $\eta$) that
$$
|\eta'_a(k)| \le C |k|^{a} \mand |\eta_a''(k)| \le C |k|^{a-1}.
$$
So we conclude
that
\be\label{qH0}
|\lambda_a''(k)-{1 \over 12} \theta_a''''(0)| \le C |k|^{a-3}.
\ee
Note that we see from this calculation that $\lambda''(0) =\theta''''(0)/12$.
Furthermore, if $a \in (4,5)$ all of the above can be repeated with one more integral and one more derivative in 
\eqref{eta int} and we wind up with \eqref{qH0} unchanged.  Likewise, if $a > 5$ you wind up once again at \eqref{qH0}
but the right hand side is $Ck^2$. And for $a = 5$ you have $C |k|^{2-\delta}$ for any $\delta > 0$.

Next, for $a \in (3,5)$ if we use the FTOC again, followed by the previous estimate:
\be\begin{split}
\left| \lambda_a(k) - \lambda_a(0) - {1 \over 2} \lambda_a''(0) k^2
\right|
 &= \left \vert \int_0^k \int_0^y [\lambda_a''(z)-\lambda_a''(0)] dz dy \right \vert\\& \le C\left \vert \int_0^k \int_0^y |z|^{a-3} dz dy \right \vert\\& \le 
C|k|^{a-1}.
\end{split}
\ee
Thus the second estimate in (iii) holds
with $\sigma = a-3$, provided $a \in (3,5)$.   If $a > 5$ one has the same but with $\sigma = 2$.  And
for $a = 5$ you can take $\sigma = 2-\delta$, where $\delta >0$ is arbitrary.

Then note that by taking a small enough value of $k_*$, the first estimate in (iii) is a byproduct of the second estimate and (ii). So now we have all of (iii).  As for (iv), recall the $\al_m >0$ which, as we saw in Section \ref{mh section}, implies
$\lambda(k) < \lambda(0)$ for all $k$.  As in that section, the continuity of $\lambda(k)$ and its convergence to zero implies (iv).  
Thus we have all the hypotheses necessary to deploy Theorem \ref{type I sw}, with $\sigma$ taken as in the previous paragraph. That is we have:
\begin{cor} Let $\Phi_m(r)=\Phi(r) = 1/r^a$ and $r_* = 1$. For all $a > 3$ there exist $\ep_a>0$, $\kappa_a > 0$ 
and $\sigma_a>0$ so that the following hold for $\ep \in (0,1)$.  There is a unique function $V_\ep \in E^3$ with
$\|V_\ep\|_{H^1} \le \kappa_a$ so that 
$$
c_\ep^2 = a(a+1)\zeta(a) + {a(a+1) \over 12}\zeta(a-2) \ep^2 \mand W_\ep(x) = -{\zeta(a-2) \over 4 (a+2) \zeta(a)} 
\sech^2(x)+ \ep^{\sigma_a} V_\ep(x)
$$
 solve \eqref{TWE1}. If $a \ne 5$ then $\sigma_a = \min\{a-3,2 \}$.  If $a = 5$, then 
 $\sigma_a = 2-\delta$ for any $\delta > 0$.
\end{cor}

\begin{remark} In the above $\ds \zeta(a) := \sum_{m \ge 1} m^{-a}$ is the famous zeta-function.  
The profile $W_\ep(x)$ is smooth because $1/r^a$ is smooth.  Additionally, one can show that $W_\ep(x)$ is  negative
for all $x$.  The argument is exactly the same as the one which is used in \cite{IP2} to establish the positivity 
of traveling waves in the case $a \in (4/3,3)$ and so we leave it out.  Lastly, we note that the case $a =3$ remains open; formal estimates from \cite{IP1} indicate that there should be a KdV-like solitary wave solution.  But the method presented here is insufficient in its present form to establish this.

\end{remark}

{\bf Acknowledgements:} J.D.W would like to thank the NSF who funded this work under grant DMS-2006172 and also the Simons Foundation, under MPS-TSM-00007725.

\bibliographystyle{plain}
\bibliography{lr}{}

\end{document}